\documentclass{article}

\usepackage{etex, stmaryrd, epsfig,
  graphics, psfrag, latexsym, mathtools, mathrsfs,enumitem,
  longtable,booktabs,ifthen}
\usepackage{amsfonts, amsmath, amsthm, amssymb}
\usepackage[noadjust]{cite}
\usepackage{xcolor}
\usepackage[all,cmtip]{xy}
\usepackage{tikz} \usetikzlibrary{matrix,arrows,shapes,calc}
\definecolor{darkblue}{rgb}{0,0,0.4} 
\usepackage{xr-hyper}
\usepackage[colorlinks=true, citecolor=darkblue, filecolor=darkblue, linkcolor=darkblue,urlcolor=darkblue]{hyperref} 
\usepackage[all]{hypcap}
\newcommand{\Sadey}{\begin{tikzpicture}[baseline=-\the\dimexpr\fontdimen22\textfont2\relax] \draw (0,0) circle (1ex); \node at (50:0.4ex) {.}; \node at (130:0.4ex) {.}; \draw (-35:0.75ex) arc (35:145:0.75ex);\end{tikzpicture}}
\newcommand{\Quesy}{\begin{tikzpicture}[baseline=-\the\dimexpr\fontdimen22\textfont2\relax] \draw (0,0) circle (1ex); 
\node at (0ex,-0.52ex) {.};
\draw[line width=0.7pt] plot[smooth] coordinates {(0ex,-0.2ex) (0ex,0ex) (0.2ex,0.25ex) (0.2ex,0.5ex) (0ex,0.6ex) (-0.25ex,0.55ex) (-0.3ex,0.4ex)};
\fill[black] (-0.3ex,0.4ex) circle (0.35pt);
\fill[black] (0ex,-0.2ex) circle (0.35pt);
\end{tikzpicture}}

\graphicspath{{draws/}{}}

\numberwithin{equation}{section}

\newtheorem{thm}{Theorem}
\newtheorem{theorem}[thm]{Theorem}

\newtheorem{lem}{Lemma}[section]

\newtheorem{corollary}[lem]{Corollary}

\theoremstyle{definition}

\theoremstyle{remark}

\numberwithin{figure}{section}
\numberwithin{table}{section}

\newcommand{\Smod}{\mathbf}

\newcommand{\R}{\mathbb{R}}

\newcommand{\F}{\mathbb{F}}
\newcommand{\C}{\mathbb{C}}

\newcommand{\N}{\mathbb{N}}

\newcommand{\QQ}{\mathbb{Q}}

\newcommand{\mc}{\mathcal}

\newcommand{\wt}{\widetilde}
\newcommand{\ol}{\overline}

\renewcommand{\emptyset}{\varnothing}

\newcommand{\smas}{\wedge}

\newcommand{\from}{\nobreak\mskip2mu\mathpunct{}\nonscript
  \mkern-\thinmuskip{:}\penalty300\mskip6muplus1mu\relax}
\newcommand{\into}{\hookrightarrow}
\newcommand{\too}{\longrightarrow}

\renewcommand{\th}{^{\text{th}}}

\renewcommand{\hat}{\widehat}

\DeclareMathOperator{\Ob}{Ob}
\DeclareMathOperator{\Id}{Id}
\DeclareMathOperator{\Sq}{Sq}

\DeclareMathOperator{\Hom}{Hom}

\newcommand{\Kh}{\mathit{Kh}}
\newcommand{\rKh}{\widetilde{\Kh}}
\newcommand{\KhCx}{\mc{C}_{\mathit{Kh}}}

\newcommand{\Cat}{\mathscr{C}}

\newcommand{\Codim}[1]{\langle#1\rangle}

\newcommand{\Moduli}{\mathcal{M}}

\newcommand{\gr}{\mathrm{gr}}

\newcommand{\KhSpace}{\mathcal{X}_\mathit{Kh}}
\newcommand{\rKhSpace}{\widetilde{\mathcal{X}}_\mathit{Kh}}

\newcommand{\ChainCat}[1]{[#1]}

\newcommand{\co}{\from}
\newcommand{\bdy}{\partial}
\newcommand{\RR}{\R}
\newcommand{\FF}{\F}
\newcommand{\CC}{\C}

\newcommand{\pt}{\mathrm{pt}}

\newcommand{\NN}{\N}
\newcommand{\ZZ}{\mathbb{Z}}

\DeclareMathOperator{\image}{im}

\newcommand{\mathcenter}[1]{\vcenter{\hbox{$#1$}}}

\newcommand{\cellC}[1][\bullet]{\wt{C}^{\mathrm{cell}}_{#1}}
\newcommand{\cocellC}[1][\bullet]{\wt{C}_{\mathrm{cell}}^{#1}}
\DeclareMathOperator{\Cone}{Cone}

\newcommand{\basis}[1]{\langle #1\rangle}
\newcommand{\Burn}[1]{\mathscr{B}(#1)}
\newcommand{\iBurn}[1]{\mathscr{T}(#1)}
\newcommand{\BurnsideCat}{\mathscr{B}}

\newcommand{\AbFunc}{\mathsf{Ab}}

\newcommand{\Complexes}{\mathsf{Kom}}
\newcommand{\Top}{\mathsf{Top}}
\newcommand{\mTop}{\underline{\Top}}
\newcommand{\chain}{\mathsf{ch}}
\newcommand{\tot}{\mathsf{Tot}}

\newcommand{\CW}{\mathsf{CW}}

\newcommand{\Sets}{\mathsf{Sets}}

\newcommand{\SphereS}{\mathbb{S}} 

\newcommand{\Cob}{\mathsf{Cob}}

\newcommand{\lax}[1]{\mathfrak{h}{#1}}

\newcommand{\norm}[1]{\Vert #1\Vert}

\newcommand{\LL}{\mathbb{L}}

\tikzstyle{crossing}=[circle,fill=white,minimum height=6pt,inner sep=0pt, outer sep=0pt, style={transform shape=false}]

\newcommand{\rsusp}{\Sigma}

\definecolor{darkgreen}{rgb}{0,.3,0}
\definecolor{darkred}{rgb}{0.3,0,0}

\begin{document}
\title{Spatial refinements and Khovanov homology}

\author{Robert Lipshitz\thanks{Department of Mathematics, University
    of Oregon, Eugene, OR 97403.} \and Sucharit Sarkar\thanks{Department of
    Mathematics, University of California, Los Angeles, CA
    90095. \newline \indent\hspace{.45em} RL was supported by NSF
    DMS-1642067. SS was supported by NSF DMS-1643401.}}


\makeatletter
\newcommand{\subjclass}[2][1991]{%
  \let\@oldtitle\@title%
  \gdef\@title{\@oldtitle\footnotetext{#1 \emph{Mathematics subject classification.} #2}}%
}
\newcommand{\keywords}[1]{%
  \let\@@oldtitle\@title%
  \gdef\@title{\@@oldtitle\footnotetext{\emph{Key words and phrases.} #1.}}%
}
\makeatother

\subjclass[2010]{\href{http://www.ams.org/mathscinet/search/mscdoc.html?code=57M25,55P42}{57M25,55P42}}



\date{}

\maketitle

\begin{abstract}
  We review the construction and context of a stable homotopy refinement of Khovanov homology.
\end{abstract}



\section{Introduction}
While studying critical points and
geodesics, \cite{Morse-top-fns,Morse-top-Morse2,Morse-top-book}
introduced what is now called \emph{Morse theory}---using functions
for which the second derivative test does not fail (\emph{Morse
  functions}) to decompose manifolds into simpler pieces. The
finite-dimensional case was further developed by many authors
(see~\cite{Bott-top-history} for a survey of the history), and an
infinite-dimensional analogue introduced by
\cite{PS-top-Bull,Palais-top-PS,Smale-top-PS}. In both
cases, a Morse function $f$ on $M$ leads to a chain complex $C_*(f)$
generated by the critical points of $f$. This chain complex satisfies
the \emph{fundamental theorem of Morse homology}: its homology
$H_*(f)$ is isomorphic to the singular homology of $M$. This is both a
feature and a drawback: it means that one can use information about
the topology of $M$ to deduce the existence of critical points of $f$,
but also implies that $C_*(f)$ does not see the smooth topology of
$M$. (See~\cite{Milnor-top-Morse,Milnor-top-h} for an
elegant account of the subject's foundations and some of its applications.)

Much later, \cite{Floer-top-unregularized,Floer-gauge-instanton,Floer-top-Lagrangian} introduced several new examples of infinite-dimensional, Morse-like theories. Unlike Palais-Smale's Morse theory, in which the descending manifolds of critical points are finite-dimensional, in Floer's setting both ascending and descending manifolds are infinite-dimensional. Also unlike Palais-Smale's setting, Floer's homology groups are not isomorphic to singular homology of the ambient space (though the singular homology acts on them). Indeed, most Floer (co)homology theories seem to have no intrinsic cup product operation, and so are unlikely to be the homology of any natural space.

\cite{CJS-gauge-floerhomotopy} proposed that although Floer homology is not the homology of a space, it could be the homology of some associated spectrum (or pro-spectrum), and outlined a construction, under restrictive hypotheses, of such an object. While they suggest that these spectra might be determined by the ambient, infinite-dimensional manifold together with its polarization (a structure which seems ubiquitous in Floer theory), their construction builds a CW complex cell-by-cell, using the moduli spaces appearing in Floer theory. (We review their construction in \S\ref{sec:flow-cat}. Steps towards describing Floer homology in terms of a polarized manifold have been taken by \cite{Lipy-top-geom}.) Although the Cohen-Jones-Segal approach has been stymied by analytic difficulties, it has inspired other constructions of stable homotopy refinements of various Floer homologies and related invariants; see \cite{Furuta-gauge-BF,BF-gauge-BF,Bauer-gauge-BF,Man-gauge-swspectrum,KM-gauge-swspectrum,Douglas-top-spectre,Cohen10:Floer-htpy-cotangent,Cohen-gauge-realize,Kragh:transfer-spectra,Kragh-gauge-param,AK-gauge-immersion,Khandhawit-gauge-slice,Khandhawit-gauge-Conley,Sasahira-gauge-gluing,TLS-gauge-unfolded}.

From the beginning, Floer homologies have been used to define
invariants of objects in low-dimensional topology---3-manifolds,
knots, and so on. In a slightly different direction,
\cite{Kho-kh-categorification} defined another knot invariant, which
he calls \emph{$\mathfrak{sl}_2$ homology} and everyone else calls
\emph{Khovanov homology}, whose graded Euler characteristic is the
Jones polynomial from \cite{Jones-polynomial}. (See
\cite{Bar-kh-khovanovs} for a friendly introduction.) While it looks
formally similar to Floer-type invariants, Khovanov homology is
defined combinatorially. No obvious infinite-dimensional manifold or
functional is present. Still, \cite{SS-Kh-symp} (inspired by earlier
work of \cite{KS-kh-braid} and others) gave a conjectural
reformulation of Khovanov homology via Floer homology. Over $\QQ$, the
isomorphism between Seidel-Smith's and Khovanov's invariants was
recently proved by \cite{AS-kh-agree-Q}. \cite{Man-kh-symp-sln} gave
an extension of the reformulation to $\mathfrak{sl}_n$ homology
constructed by \cite{KR-kh-matrixfactorizations}.

Inspired by this history, \cite{RS-khovanov,RS-steenrod,RS-s-invariant} gave a combinatorial definition of a
spectrum refining Khovanov homology, and studied some of its
properties. This circle
of ideas was further developed in \cite{LNS-Plam} and in
\cite{LLS-khovanov-product,LLS-tangles}, and extended in many
directions by other authors (see \S\ref{sec:prop-app}).  Another
approach to a homotopy refinement was given by \cite{ET-kh-spectrum},
though it turns out their invariant is determined by Khovanov
homology, cf.~\cite{ELST-trivial}. Inspired by a different line of
inquiry, \cite{HKK-Kh-htpy} also gave a construction of a Khovanov
stable homotopy type. \cite{LLS-khovanov-product} show that the two
constructions give homotopy equivalent spectra,
perhaps suggesting some kind of uniqueness.

Most of this note is an outline of a construction of a Khovanov homotopy type, following \cite{LLS-khovanov-product} and \cite{HKK-Kh-htpy}, with an emphasis on the general question of stable homotopy refinements of chain complexes. In the last two sections, we briefly outline some of the structure and uses of the homotopy type (\S\ref{sec:prop-app}) and some questions and speculation (\S\ref{sec:speculation}). Another exposition of some of this material can be found in~\cite{LLS-cube}.

\emph{Acknowledgments.} We are deeply grateful to our collaborator
Tyler Lawson, whose contributions and perspective permeate the account
below. We also thank Mohammed Abouzaid, Ciprian Manolescu, and John Pardon
for comments on a draft of this article.

\section{Spatial refinements}

The spatial refinement problem can be summarized as follows.

\emph{Start with a chain complex $C_*$ with a distinguished, finite
  basis, arising in some interesting setting.  Incorporating more
  information from the setting, construct a based CW complex (or
  spectrum) whose reduced cellular chain complex, after a shift, is
  isomorphic to $C_*$ with cells corresponding to the given basis.}

A result of~\cite{Carlsson-top-counter} implies that there
is no universal solution to the spatial refinement problem, i.e., no
functor $S$ from chain complexes (supported in large gradings, say) to
CW complexes so that the composition of $S$ and the reduced cellular
chain complex functor is the identity
(cf.~\cite{nLab-nonfunc}). Specifically, for $G=\ZZ/2\times \ZZ/2$ he
defines a module $P$ over $\ZZ[G]$ so that there is no $G$-equivariant
Moore space $M(P,n)$ for any $n$.  If $C_*$ is a free resolution of
$P$ over $\ZZ[G]$ then $S(C_*)$ would be such a Moore space, a
contradiction.

Thus, spatially refining $C_*$ requires context-specific work.  This
section gives general frameworks for such spatial refinements, and the
next section has an interesting example of one.

\subsection{Linear and cubical diagrams}

Let $C_*$ be a freely and finitely generated chain complex with a
given basis. After shifting we may assume
$C_*$ is supported in gradings $0,\dots,n$. Let $\ChainCat{n+1}$ be the category
with objects $0,1,\dots,n$ and a unique morphism $i\to j$ if
$i\geq j$.
Let $\Burn{\ZZ}$ denote the category of finitely
generated free abelian groups,
with objects finite
sets and $\Hom_{\Burn{\ZZ}}(S,T)$ the set of linear maps
$\ZZ\basis{S}\to\ZZ\basis{T}$ or, equivalently, $T\times S$ matrices of integers.
Then $C_*$ may be viewed as a functor $F$
from $\ChainCat{n+1}$ to $\Burn{\ZZ}$ subject to the condition that $F$ sends any
length two arrow (that is, a morphism $i\to j$ with $i-j=2$) to the
zero map. Given such a functor $F\from\ChainCat{n+1}\to\Burn{\ZZ}$, that
is, a \emph{linear diagram}
\begin{equation}
\ZZ\basis{F(n)}\to\ZZ\basis{F(n-1)}\to\dots\to\ZZ\basis{F(1)}\to\ZZ\basis{F(0)}
\end{equation}
with every composition the zero map, we obtain a chain complex $C_*$
by shifting the gradings, $C_i=\ZZ\basis{F(i)}[i]$, and letting $\bdy_i=F(i\to i-1)$.
This construction is functorial. That is, if $\Burn{\ZZ}^{\ChainCat{n+1}}_{\bullet}$
denotes the full subcategory of the functor category
$\Burn{\ZZ}^{\ChainCat{n+1}}$ generated by those functors
which send every length two arrow to the zero map, and if $\Complexes$
denotes the category of chain complexes,
then the above construction is a functor
$\chain\from\Burn{\ZZ}^{\ChainCat{n+1}}_{\bullet}\to\Complexes$. Indeed,
it would be reasonable to call an element of
$\Burn{\ZZ}^{\ChainCat{n+1}}_{\bullet}$ a \emph{chain complex in
  $\Burn{\ZZ}$.}

A linear diagram
$F\in\Burn{\ZZ}^{\ChainCat{n+1}}_{\bullet}$ may also be viewed
as a \emph{cubical diagram} $G\from\ChainCat{2}^n\to\Burn{\ZZ}$ by setting
\begin{equation}\label{eq:linear-to-cube}
  G(v)=\begin{cases}
    F(i)&\text{if $v=(\underbrace{0,\dots,0}_{n-i},\underbrace{1,\dots,1}_{i})$}\\
    \emptyset&\text{otherwise.}
  \end{cases}
\end{equation}
On morphisms, $G$ is either zero or induced from $F$, as appropriate.
Conversely, a cubical diagram $G\in\Burn{\ZZ}^{\ChainCat{2}^n}$
gives a linear diagram
$F\in\Burn{\ZZ}^{\ChainCat{n+1}}_{\bullet}$ by setting
$F(i)=\coprod_{\substack{|v|=i}}G(v)$,
where $|v|$ denotes the number of $1$'s in $v$.
The component of $F(i+1\to i)$ from
$\ZZ\basis{G(u)}\subset\ZZ\basis{F(i+1)}$ to
$\ZZ\basis{G(v)}\subset\ZZ\basis{F(i)}$ is 
\begin{equation}
\begin{cases}
(-1)^{u_1+\dots+u_{k-1}}G(u\to v)&\text{if $u-v=\hat{e}_k$, the $k\th$ unit vector,}\\
0&\text{if $u-v$ is not a unit vector.}
\end{cases}
\end{equation}
These give functors
\(\vcenter{\hbox{
\begin{tikzpicture}[xscale=3,yscale=0.1]
\node (a) at (0,0) {$\Burn{\ZZ}^{\ChainCat{n+1}}_{\bullet}$};
\node (b) at (1,0) {$\Burn{\ZZ}^{\ChainCat{2}^n}$};
\draw[->] ($(a.east)+(0,1)$)--($(b.west)+(0,1)$) node[midway,anchor=south] {\small $\alpha$};
\draw[<-] ($(a.east)+(0,-1)$)--($(b.west)+(0,-1)$) node[midway,anchor=north] {\small $\beta$};
\end{tikzpicture}
}}\) with $\beta\circ\alpha=\Id$. 

The composition $\chain\circ\beta\co
\Burn{\ZZ}^{\ChainCat{2}^n}\to \Complexes$ is the \emph{totalization}
$\tot$, and may be viewed as an iterated mapping cone.
Up to chain homotopy equivalence, one can also construct $\tot$ using
homotopy colimits.  Define a category $\ChainCat{2}_+$ by adjoining a
single object $*$ to $\ChainCat{2}$ and a single morphism $1\to *$;
let $\ChainCat{2}^n_+=(\ChainCat{2}_+)^n$. Given
$G\in\Burn{\ZZ}^{\ChainCat{2}^n}$, by treating abelian groups as chain
complexes supported in homological grading zero, we get an associated
cubical diagram $A\from\ChainCat{2}^n\to\Complexes$. Extend $A$ to a
diagram $A_+\from \ChainCat{2}_+^n\to\Complexes$ by setting
\begin{equation}\label{eq:enlarge-dia}
  A_+(v)=\begin{cases}
    A(v)&\text{if $v\in\ChainCat{2}^n$}\\
    0&\text{otherwise.}
  \end{cases}
\end{equation}
Then the totalization of $G$ is the \emph{homotopy colimit} of 
$A_+$. (See~\cite{Segal-top-categories,BK-top-book,Vogt-top-hocolim}.)

\subsection{Spatial refinements of diagrams of abelian groups}

As a next step, given a finitely generated chain complex represented
by a functor $F\from \ChainCat{n+1}\to\Burn{\ZZ}$
we wish to construct a based cell complex with cells in
dimensions $N,\dots,N+n$ whose reduced cellular complex---with
distinguished basis given by the cells---is isomorphic to the given
complex shifted up by $N$. 

Let
$\iBurn{S^N}$ denote the category with objects finite sets and
morphisms $\Hom_{\iBurn{S^N}}(S,T)$ the set of all based maps
$\bigvee_S S^N\to \bigvee_T S^N$ between wedges of $N$-dimensional
spheres; applying reduced $N\th$ homology to the morphisms produces a
functor, also denoted $\wt{H}_N$, from $\iBurn{S^N}$ to
$\Burn{\ZZ}$. A \emph{strict $N$-dimensional spatial lift} of $F$ is a
functor $P\from\ChainCat{n+1}\to\iBurn{S^N}$ satisfying $\wt{H}_N\circ
P=F$ and $P$ is the constant map on any length two arrow in
$\ChainCat{n+1}$, i.e., a \emph{strict chain complex in $\iBurn{S^N}$
  lifting $F$}. Just as morphisms in $\Burn{\ZZ}$ are matrices, if we
replace $S^N$ by the sphere spectrum $\SphereS$, we
may view a morphism in $\iBurn{\SphereS}$ as a matrix of maps
$\SphereS\to \SphereS$ by viewing $\bigvee_S \SphereS$ as a
coproduct and $\bigvee_T \SphereS$ as a product.

Given such a linear diagram $P$, we can construct a based cell complex
by taking mapping cones and suspending sequentially, cf.~\cite[\S
5]{CJS-gauge-floerhomotopy}. If $\CW$ denotes the category of based
cell complexes, then $P$ induces a diagram
$X\from\ChainCat{n+1}\to\CW$,
\begin{equation}
  X(n)\stackrel{f_n}{\too}
  X(n-1)\stackrel{f_{n-1}}{\too}\dots\stackrel{f_{2}}{\too}X(1)\stackrel{f_{1}}{\too}X(0)
\end{equation}
with every composition the constant map. Since $f_{1}\circ f_{2}$
is the constant map, there is an induced map
$g_{1}\co \rsusp X(2)\to\Cone(f_{1})$ from the reduced suspension to the
reduced cone. Then we get a diagram $Y\from\ChainCat{n}\to\CW$,
\begin{equation}
  Y(n-1)=\rsusp X(n)\stackrel{\rsusp f_n}{\too}\dots\stackrel{\rsusp f_3}{\too}
  Y(1)=\rsusp X(2)\stackrel{g_1}{\too}Y(0)=\Cone(f_{n}).
\end{equation}
Take the mapping cone of $g_1$ and suspend to get a diagram
$Z\from\ChainCat{n-1}\to\CW$ and so on. 
The reduced cellular chain
complex of the final CW complex is the original chain complex, shifted
up by $N$. This construction is also functorial: if
$\iBurn{S^N}^{\ChainCat{n+1}}_{\bullet}$ is the full subcategory
generated by the functors which send every length two arrow to the
constant map, then the construction is a functor
$\iBurn{S^N}^{\ChainCat{n+1}}_{\bullet}\to\CW$. The construction can also be carried out in a single step.
Construct a category $\ChainCat{n+1}_+$ by adjoining a single object
$*$ and a unique morphism $i\to *$ for all $i\neq 0$.  Extend
$X\from\ChainCat{n+1}\to\CW$ to 
$X_+\from\ChainCat{n+1}_+\to\CW$ by sending $*$ to a point and
take the homotopy colimit of $X_+$.

A linear diagram
$P\in\iBurn{S^N}^{\ChainCat{n+1}}_{\bullet}$ produces a cubical
diagram $Q\from\ChainCat{2}^n\to\iBurn{S^N}$ by 
the analogue of Equation~(\ref{eq:linear-to-cube}).
There is a \emph{totalization} functor
$\tot\from\iBurn{S^N}^{\ChainCat{2}^n}\to\CW$ extending the functor
$\iBurn{S^N}^{\ChainCat{n+1}}_{\bullet}\to\CW$ so that 
\begin{equation}\label{eq:tot-commutes}
\vcenter{\hbox{\begin{tikzpicture}[xscale=2.5,yscale=-1.2]
\node (chainspace) at (-0.3,0) {$\iBurn{S^N}^{\ChainCat{n+1}}_{\bullet}$};
\node (chaingroup) at (-0.3,1) {$\Burn{\ZZ}^{\ChainCat{n+1}}_{\bullet}$};
\node (cell) at (2,0) {$\CW$};
\node (chcx) at (2,1) {$\Complexes.$};
\node (cubespace) at (1,0) {$\iBurn{S^N}^{\ChainCat{2}^n}$};
\node (cubegroup) at (1,1) {$\Burn{\ZZ}^{\ChainCat{2}^n}$};

\draw[->] (chainspace) -- (chaingroup);
\draw[->] (cubespace) -- (cubegroup);
\draw[->] (chainspace) -- (cubespace);
\draw[->] (chaingroup) -- (cubegroup);
\draw[->] (cubespace) -- (cell);
\draw[->] (cubegroup) -- (chcx);
\draw[->] (cell) -- (chcx) node[midway,anchor=west] {\small $\cellC[*][-N]$};
\end{tikzpicture}}}
\end{equation}
commutes.
The totalization functor is defined as an iterated mapping cone
or as a homotopy colimit of an extension of $Q$ analogous to
Equation~(\ref{eq:enlarge-dia}). 

\subsection{Lax spatial refinements}

Instead of working with strict functors as in the previous section,
sometimes is it more convenient to work with
lax functors. A \emph{lax} or \emph{homotopy coherent} or \emph{$(\infty,1)$} \emph{functor} $F\from\Cat\to\Top$ is a diagram
that commutes up to homotopies which are specified, and the homotopies
themselves commute up to higher homotopies which are also specified,
and so on; for details see \cite{Vogt-top-hocolim,Cordier,Lurie-top-htt}. More precisely, $F$ consists of 
based topological spaces $F(x)$ for $x\in\Cat$,
  and higher homotopy maps $F(f_n,\dots,f_1)\from [0,1]^{n-1}\times
  F(x_0)\to F(x_n)$ for composable morphisms
  $x_0\stackrel{f_1}{\too}\dots\stackrel{f_n}{\too} x_n$ with certain
  boundary conditions and restrictions involving basepoints and
  identity morphisms;
  the case $n=1$ is maps corresponding to the arrows in the diagram,
  $n=2$ is homotopies corresponding to pairs of composable arrows,
  etc.  A strict functor may be viewed as a lax functor. Let
  $\lax{\Top}^{\Cat}$ denote the category of lax functors
  $\Cat\to\Top$, with morphisms given by lax functors
  $\Cat\times\ChainCat{2}\to\Top$. (There are also higher morphisms
  corresponding to lax functors $\Cat\times\ChainCat{n}\to\Top$.)

There is a notion of a lax functor to $\iBurn{S^N}$ induced from the
notion of lax functors to $\Top$. Let
$\lax{\iBurn{S^N}}^{\ChainCat{n+1}}_{\bullet}$ be the subcategory of
$\lax{\iBurn{S^N}}^{\ChainCat{n+1}}$ consisting of those objects
(respectively, morphisms) $F$ such that
$F\bigl(x_0\stackrel{f_1}{\too}\dots\stackrel{f_n}{\too} x_n\bigr)$ is
the constant map to the basepoint for any string of morphisms in
$\ChainCat{n+1}$ (respectively, $\ChainCat{n+1}\times\ChainCat{2}$)
with some $f_i$ of length $\geq 2$ in $\ChainCat{n+1}$
(respectively, $\ChainCat{n+1}\times\{0,1\}$).  We call such functors
\emph{chain complexes in $\iBurn{S^N}$}. (In Cohen-Jones-Segal's
language, chain complexes in $\iBurn{S^N}$ are functors
$\mathscr{J}_0^n\to\mathscr{T}_*$.)

If one starts with a chain complex $F\in\Burn{\ZZ}^{\ChainCat{n+1}}_\bullet$
and wishes to refine it to a based cell complex,
instead of constructing a strict $N$-dimensional spatial lift in
$\iBurn{S^N}^{\ChainCat{n+1}}_{\bullet}$, it is enough to construct a
\emph{lax $N$-dimensional spatial lift}, that is, a functor
$P\in\lax{\iBurn{S^N}}^{\ChainCat{n+1}}_{\bullet}$ with
$\wt{H}_N\circ P=F$.
Such a $P$ produces a
cell complex by adjoining a basepoint to get a lax diagram
$\ChainCat{n+1}_+\to\CW$ and then taking homotopy
colimits.  Alternatively, we may convert $P$ to a lax cubical diagram
$Q\in\lax{\iBurn{S^N}}^{\ChainCat{2}^n}$ and proceed as before. The
iterated mapping cone construction becomes intricate since the
associated diagram $X\from\ChainCat{2}^n\to\CW$ is lax. So, extend to a lax diagram
$X_+\from\ChainCat{2}^n_+\to\CW$ as before and then take its homotopy
colimit.
This
generalization to the lax set-up
remains functorial and the analogue of Diagram~\eqref{eq:tot-commutes} still commutes.

\subsection{Framed flow categories}\label{sec:flow-cat}
\cite{CJS-gauge-floerhomotopy} first proposed
lax spatial refinements of diagrams
$F\from\ChainCat{n+1}\to\Burn{\ZZ}$ via
framed flow categories, using the Pontryagin-Thom
construction. A \emph{framed flow
  category} is an abstraction of the gradient
flows of a Morse-Smale function. Concretely, a framed
flow category $\Cat$ consists of:
\begin{enumerate}[leftmargin=*,parsep=2pt]
\item A finite set of \emph{objects} $\Ob(\Cat)$ and a \emph{grading}
  $\gr\co \Ob(\Cat)\to\ZZ$. After translating, we may assume the gradings
  lie in $[0,n]$. 
\item For $x,y\in\Cat$ with $\gr(x)-\gr(y)-1=k$, a \emph{morphism set}
  $\Moduli(x,y)$ which is a $k$-dimensional $\Codim{k}$-manifold. A
  \emph{$\Codim{k}$-manifold} $M$ is a smooth manifold with corners so
  that each codimension-$c$ corner point lies in exactly $c$ facets
  (closure of a codimension-$1$ component), equipped with a
  decomposition of its boundary $\bdy M=\cup_{i=1}^k\bdy_i M$ so that
  each $\bdy_i M$ is a \emph{multifacet} of $M$ (union of disjoint
  facets), and $\bdy_iM\cap\bdy_jM$ is a multifacet of $\bdy_iM$
  and $\bdy_jM$, cf.~\cite{Jan-top-o(n)manifolds,Lau-top-cobordismcorners}.
\item An associative \emph{composition} map
  $\Moduli(y,z)\times\Moduli(x,y)\into\bdy_{\gr(y)-\gr(z)}\Moduli(x,z)\subset\Moduli(x,z)$. Setting 
  \begin{equation}
  \Moduli(i,j)=\coprod_{\substack{x,y\\\gr(x)=i,\gr(y)=j}}\Moduli(x,y),
  \end{equation}
  the composition is required to induce an isomorphism of
  $\Codim{i-j-2}$-manifolds
  \begin{equation}
  \bdy_{j-k}\Moduli(i,k)\cong \Moduli(j,k)\times\Moduli(i,j).
  \end{equation}
\item \emph{Neat embeddings}
  $\iota_{i,j}\from\Moduli(i,j)\into[0,1)^{i-j-1}\times(-1,1)^{D(i-j)}$
  for some large $D\in\NN$,
  namely, smooth embeddings satisfying
  \begin{equation}
  \iota^{-1}_{i,k}([0,1)^{j-k-1}\times\{0\}\times[0,1)^{i-j-1}\times(-1,1)^{D(i-k)})=\bdy_{j-k}\Moduli(i,k)
  \end{equation}
  and certain orthogonality conditions near boundaries. These
  embeddings are required to be coherent with respect to composition.
   The
   space of such collections of neat embeddings is
   $(D-2)$-connected.
 \item \emph{Framings} of the normal bundles of $\iota_{i,j}$, also
   coherent with respect to composition, which give
   extensions 
   $\ol{\iota}_{i,j}\from \Moduli(i,j)\times[-1,1]^{D(i-j)}\into[0,1)^{i-j-1}\times(-1,1)^{D(i-j)}$.
\end{enumerate}

\begin{figure}[!t]
  \centering
  \begin{tikzpicture}[scale=1.4]
    \foreach\i in {0,1,2}{
      \begin{scope}[xshift=3*\i cm]
        \draw[->] (-1,0) -- (1,0) node[pos=0.5,anchor=north]
        {\scriptsize $(-1,1)$} ;
      \end{scope}
    }

    \pgfmathsetmacro{\a}{-0.1}
    \pgfmathsetmacro{\b}{-0.5}
    \pgfmathsetmacro{\c}{0.5}
    \pgfmathsetmacro{\d}{0.7}

    \foreach \l/\xs/\s in {a/0/-1,b/1/1,c/1/-1,d/2/1}{
      \begin{scope}[xshift=\xs*3cm]
        \draw[->,thick] (\expandafter\csname \l\endcsname cm-\s*0.2cm,0)--++(\s*0.4cm,0) node[midway] {\tiny $\bullet$} node[midway, anchor=south] {\scriptsize $\l$};
      \end{scope}
    }

    \foreach\i in {0,1}{
      \begin{scope}[xshift=3*\i cm,yshift=-2 cm]
        \draw[->] (0,0) -- (0,1) node[pos=1,anchor=south east,inner sep=0pt, outer sep=0pt] {\scriptsize $[0,1)$};
        \begin{scope}[cm={1,0,0.3,0.8,(0cm,0cm)}]
          \draw[->] (-1,0) -- (1,0) node[pos=1,anchor=north west,inner sep=0pt, outer sep=0pt] {\scriptsize $(-1,1)$};
          \draw[->] (0,-1) -- (0,1) node[pos=1,anchor=south west,inner sep=3pt, outer sep=0pt] {\scriptsize $(-1,1)$};
        \end{scope}
      \end{scope}
    }

    \foreach \l/\k/\xs/\ls/\ks/\an in {b/a/0/1/-1/north east,c/a/0/-1/-1/north west,d/b/1/1/1/north west,d/c/1/1/-1/north west}{
      \begin{scope}[cm={1,0,0.3,0.8,(\xs*3cm,-2cm)}]
        \node (c) at (\expandafter\csname \l\endcsname,\expandafter\csname \k\endcsname) {\tiny $\bullet$};
        \draw[->,thick] ($(c)+(-\ls*0.2,0)$)--($(c)+(\ls*0.2,0)$);
        \draw[->,thick] ($(c)+(0,-\ks*0.2)$)--($(c)+(0,\ks*0.2)$);
        \draw ($(c)-(0.2,0.2)$) rectangle ($(c)+(0.2,0.2)$);
        \node[anchor=\an,inner sep=6pt, outer sep=0pt] at ($(c)$) {\scriptsize $\l\k$};

        \foreach \n/\xcor/\ycor in {ne/1/1,nw/1/-1,se/-1/1,sw/-1/-1,cen/0/0}{
          \coordinate (\l\k\n) at ($(c)+(0.2*\ls*\xcor,0.2*\ks*\ycor)$);
        }
      \end{scope}
    }

    \begin{scope}[cm={1,0,0,1,(0cm,-2cm)}]
      \foreach \n/\r in {cen/0.5}{
        \draw[thick] (ca\n) arc (0:180:\r cm);
      }
      \foreach \n/\r in {ne/-0.2, nw/-0.2, se/0.2, sw/0.2}{
        \draw[thin] (ca\n) arc (0:180:0.5cm+\r cm);
      }
    \end{scope}

    \begin{scope}[cm={0.3,0.8,0,1,(3cm,-2cm)}]
      \foreach \n/\r in {cen/0.5}{
        \draw[thick] (dc\n) arc (0:180:\r cm);
      }
      \foreach \n/\r in {ne/-0.2, nw/0.2, se/-0.2, sw/0.2}{
        \draw[thin] (dc\n) arc (0:180:0.5cm+\r cm);
      }
    \end{scope}

    \begin{scope}[cm={1,0,0.3,0.8,(3cm,-2cm)}]
      \draw[thin] (0.51-0.2,1.12)--++(0.4,0);
      \node at (0.4,0.8) {\scriptsize $J$};
    \end{scope}
    \begin{scope}[cm={1,0,0.3,0.8,(0cm,-2cm)}]
      \draw[thin] (-0.55-0.2,0.01)--++(0,0.4);
      \node at (0.25,0.6) {\scriptsize $I$};
    \end{scope}

    \begin{scope}[xshift=3*2 cm,yshift=-2 cm]
      
      \draw[->] (0,0) -- (0,1) node[pos=1,anchor=south east,inner sep=0pt, outer sep=0pt] {\scriptsize $[0,1)$};
      \begin{scope}[cm={1,0,0.3,0.8,(0cm,0cm)}]
        \draw[->] (0,0) -- (1,0) node[pos=1,anchor=north west,inner sep=0pt, outer sep=0pt] {\scriptsize $[0,1)$};
        \draw[->] (0,-1) -- (0,1) node[pos=1,anchor=south west,inner sep=3pt, outer sep=0pt] {\scriptsize $(-1,1)^3$};

        \coordinate (btail) at (0,\b-0.2);
        \coordinate (bcen) at (0,\b);
        \coordinate (bhead) at (0,\b+0.2);

        \draw[thick,->] (0,\b-0.2)--++(0,0.4) node[midway] {\tiny $\bullet$}; 
        \draw[thick,->] (0,\c+0.2)--++(0,-0.4) node[midway] {\tiny $\bullet$}; 
        \draw[thick] (bcen) arc (-90:90:0.5cm);
        \draw[thin] (btail) arc (-90:90:0.7cm);
        \draw[thin] (bhead) arc (-90:90:0.3cm);
      \end{scope}

      \begin{scope}[cm={0.3,0.8,0,1,(0cm,0cm)}]
        \draw[thick] (bcen) arc (180:0:0.5cm);
        \draw[thin] (btail) arc (180:0:0.7cm);
        \draw[thin] (bhead) arc (180:0:0.3cm);
      \end{scope}

      \draw[thick] (0.52,0.17) to[out=90,in=0] (0.09,0.64);
      \draw[thin] (0.73,0.2) to[out=90,in=0] (0.14,0.9);
      
      \node at (-0.3,0) {\scriptsize $Ja$};
      \node at (0.5,-0.4) {\scriptsize $dI$};
      \node at (0.5,0.6) {\scriptsize $B$};
    \end{scope}
  \end{tikzpicture}
  \caption{\textbf{A framed flow category $\Cat$.}  \emph{$\Ob(\Cat)=\{x,y,z,w\}$
    in gradings $3,2,1,0$,
    respectively.  The $0$-dimensional morphism spaces are:
    $\Moduli(x,y)=\{a\}$, 
    $\Moduli(y,z)=\{b,c\}$, and
    $\Moduli(z,w)=\{d\}$, each embedded in
    $(-1,1)$. The $1$-dimensional morphism spaces are:
    $\Moduli(x,z)=I$ (resp.\ $\Moduli(y,w)=J$) an interval, embedded in $[0,1)\times(-1,1)^{2}$,
    with endpoints $\{ba,ca\}$ 
    (resp.\ $\{db,dc\}$) 
    embedded
    in $\{0\}\times(-1,1)^{2}$ by the product
    embedding. The $2$-dimensional morphism space is a
    disk $B$, embedded in $[0,1)^2\times(-1,1)^{3}$; it is a
    $\Codim{2}$-manifold with boundary decomposed as a
    union of two arcs, $\bdy_1 B=dI\subset\{0\}\times[0,1)\times(-1,1)^3$ and
    $\bdy_2 B=Ja\subset
    [0,1)\times\{0\}\times(-1,1)^3$. Coherent framings of all the normal
    bundles are represented by the
    tubular neighborhoods $\ol{\iota}_{i,j}$.  
    In the last
    subfigure, $(-1,1)^{3}$ is drawn as an interval by
    projecting to the middle $(-1,1)$.}}\label{fig:framed-flow-cat}
\end{figure}
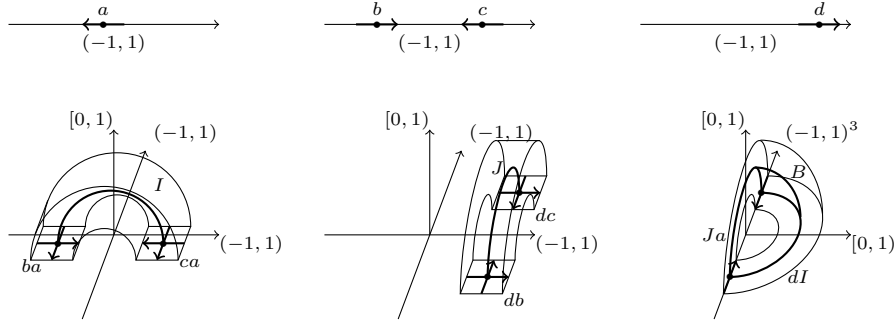

A framed flow category produces a lax linear diagram
$P\in\lax{\iBurn{S^N}}_{\bullet}^{\ChainCat{n+1}}$ with $N=nD$. On
objects, set
$P(i)=\{x\in\Cat\mid \gr(x)=i\}.$
On morphisms, define the map associated to the sequence $m_0\to
m_1\to\dots\to m_k$ to be the constant map unless all the arrows are
length one. To a sequence of length one arrows, $i\to
i-1\to\dots\to j$, associate a map
\begin{equation}
\begin{split}
[0,1]^{i-j-1}\times\bigvee_{x\in P(i)}S^N&=[0,1]^{i-j-1}\times\coprod_{x\in P(i)}[-1,1]^{nD}/\bdy\\
&\to\bigvee_{y\in P(j)}S^N=\coprod_{y\in P(j)}[-1,1]^{nD}/\bdy
\end{split}
\end{equation}
using $\ol{\iota}_{i,j}$ and the Pontryagin-Thom construction.

We can then apply the totalization functor to $P$ to get a cell
complex with cells in dimensions $N,N+1,\dots,N+n$. As
\cite{CJS-gauge-floerhomotopy} sketch, for a flow
category coming from a generic gradient flow of a Morse function,
the cell complex produced by the totalization functor is the $N\th$
reduced suspension of the Morse cell complex built from the
unstable disks of the critical points.

Much of the above data can also be reformulated in the language
of $S$-modules from \cite[\S4.3]{Pardon-top-CH}. Let $S[n+1]$ be the
(non-symmetric) multicategory with objects pairs $(j,i)$ of integers
with $0\leq j\leq i\leq n$, unique multimorphisms
$(i_0,i_1),(i_1,i_2),\allowbreak\dots,\allowbreak(i_{k-1},i_k)\to
(i_0,i_k)$ when $k\geq 1$, and no other multimorphisms (cf.~\emph{shape
  multicategories} from \cite{LLS-tangles}). Let $\mTop$ be the
multicategory of based topological spaces whose multimorphisms
$X_1,\dots,X_k\to Y$ are maps $X_1\smas\cdots\smas X_k\to Y$. An
\emph{$S$-module} is a multifunctor $S[n+1]\to\mTop$.

Given $\Cat$, define $S$-modules $\Smod{S},\Smod{J}$ by setting $\Smod{S}(j,i)=\vee_{y\in P(j)}S^{D(i-j)}$ and $\Smod{J}(j,i)=(\vee_{x\in
  P(i)}S^{D(i-j)})\smas \mathscr{J}(i,j)$ where $\mathscr{J}$ is the
category
with objects integers and morphisms $\mathscr{J}(i,j)$ the
one-point compactification of $[0,1)^{i-j-1}$ if $i\geq j$ (which is a point if $i=j$) and composition
$\mathscr{J}(j,k)\smas \mathscr{J}(i,j)\to \mathscr{J}(i,k)$ induced
by the inclusion map
$[0,1)^{j-k-1}\times \{0\}\times [0,1)^{i-j-1}\into \bdy [0,1)^{i-k-1}$,
cf.~\cite[\S5]{CJS-gauge-floerhomotopy}. On a multimorphism
$(i_0,i_1),\allowbreak\dots,\allowbreak(i_{k-1},i_k)\to (i_0,i_k)$,
$\Smod{S}$ sends the $(y_0,\dots,y_{k-1})\in P(i_0)\times\dots\times
P(i_{k-1})$ summand $S^{D(i_1-i_0)}\smas\cdots\smas
S^{D(i_k-i_{k-1})}$ homeomorphically to the $y_0$ summand
$S^{D(i_k-i_0)}$, and $\Smod{J}$ sends the $(x_1,\dots,x_{k})\in
P(i_1)\times\dots\times P(i_{k})$ summand
$(S^{D(i_1-i_0)}\smas\cdots\smas
S^{D(i_k-i_{k-1})})\smas(\mathscr{J}(i_1,i_0)\smas\cdots\smas
\mathscr{J}(i_k,i_{k-1}))$ to the $x_k$ summand
$S^{D(i_k-i_0)}\smas\mathscr{J}(i_k,i_0)$, homeomorphically on the
first factor, and using the composition in $\mathscr{J}$ on the second
factor. Given a neat embedding of $\Cat$, we can define another
$S$-module $\Smod{N}$ by setting $\Smod{N}(j,i)=\iota^{\nu}_{i,j}$,
the Thom space of the normal bundle of $\iota_{i,j}$, if $i>j$. (When $i=j$,  $\Smod{N}(j,j)$ is a point.) On multimorphisms, $\Smod{N}$ is induced from
the inclusion maps
$\image(\ol{\iota}_{i_1,i_0})\times\dots\times\image(\ol{\iota}_{i_k,i_{k-1}})\to\image(\ol{\iota}_{i_k,i_0})$. The
Pontryagin-Thom
collapse map is a
natural transformation---an \emph{$S$-module map}---from $\Smod{J}$ to
$\Smod{N}$, which sends the $x\in P(i)$ summand of $\Smod{J}(j,i)$ to
the Thom-space summand $\cup_{y\in P(j)}\iota_{x,y}^{\nu}$ in
$\Smod{N}(j,i)$. A framing of $\Cat$ produces another $S$-module map
$\Smod{N}\to\Smod{S}$ which sends the summand $\iota_{x,y}^{\nu}$ in
$\Smod{N}(j,i)$ to the $y$ summand of $\Smod{S}(j,i)$. Composing we
get an $S$-module map $\Smod{J}\to\Smod{S}$, which is precisely the
data needed to recover a lax diagram in
$\lax{\iBurn{S^N}}_{\bullet}^{\ChainCat{n+1}}$.

Finally, as popularized by Abouzaid, note that since the
(smooth) framings of $\iota_{i,j}$ were only used to construct maps
$\iota^{\nu}_{i,j}\to\vee_{y\in P(j)}S^{D(i-j)}$, a weaker structure
on the flow category---namely, coherent trivializations of the Thom
spaces $\iota^{\nu}_{i,j}$ as spherical fibrations---might suffice.

\subsection{Speculative digression: matrices of framed cobordisms}\label{sec:infty}
Perhaps it would be tidy to reformulate the notion of stably framed
flow categories as chain complexes in some category $\Burn{\Cob}$
equipped with a functor $\Burn{\Cob}\to\iBurn{\SphereS}$. It is clear
how such a definition would start. Objects in $\Burn{\Cob}$ should be
finite sets.  
By the Pontryagin-Thom construction, 
a
map $\SphereS\to\SphereS$ is determined by a framed
$0$-manifold; therefore, a morphism in $\Burn{\Cob}$ should be a
matrix of framed $0$-manifolds. To account for the homotopies in
$\iBurn{\SphereS}$, $\Burn{\Cob}$ should have higher morphisms. For
instance, given two $(T\times S)$-matrices $A,B$ of framed
$0$-manifolds, a $2$-morphism from $A$ to $B$ should be a $(T\times
S)$-matrix of framed $1$-dimensional cobordisms.  Given two such
$(T\times S)$-matrices $M,N$ of framed $1$-dimensional cobordisms, a
$3$-morphism from $M$ to $N$ should be a $(T\times S)$-matrix of
framed $2$-dimensional cobordisms with corners, and so on. That is,
the target category $\Cob$ seems to be the extended cobordism
category, an $(\infty,\infty)$-category studied, for instance, by
\cite{Lurie-top-cobordism}.

Since matrix multiplication requires only addition and multiplication,
the construction $\Burn{\Cat}$ makes sense for any \emph{rig} or
\emph{symmetric bimonoidal} category $\Cat$ and, presumably, for a rig
$(\infty,\infty)$-category, for some suitable definition; and perhaps
the framed cobordism category $\Cob$ is an example of a rig
$(\infty,\infty)$-category. Maybe the Pontryagin-Thom
construction gives a functor $\Burn{\Cob}\to \iBurn{\SphereS}$, and
that a stably framed flow category is just a functor
$\ChainCat{n+1}\to\Burn{\Cob}$.

Rather than pursuing this, we will focus on a tiny piece of $\Cob$,
in which all $0$-manifolds are framed positively, all
$1$-dimensional cobordisms are trivially-framed intervals and, more
generally, all higher cobordisms are trivially-framed disks. In this
case, all of the information is contained in the objects,
$1$-morphisms, and $2$-morphisms, and this tiny piece equals
$\Burn{\Sets}$ with $\Sets$ being viewed as a rig category via
disjoint union and Cartesian product.

\subsection{The cube and the Burnside category}\label{sec:cube-and-Burn}

The \emph{Burnside category} $\BurnsideCat$ (associated to the trivial group) is the following weak
$2$-category. The objects are finite sets. The $1$-morphisms
$\Hom(S,T)$ are $T\times S$ matrices of finite sets; composition is matrix multiplication, using
the disjoint union and product of sets in place of $+$ and $\times$ of real numbers.
The $2$-morphisms are matrices of entrywise
bijections between matrices of sets.

(The category $\BurnsideCat$ is denoted $\mathcal{S}_2$ by \cite{HKK-Kh-htpy}, and is an example of what they call a \emph{$\star$-category}. The realization procedure below is a concrete analogue of the \cite{EM-top-machine} infinite loop space machine; see also~\cite[\S8]{LLS-khovanov-product}.)

There is an abelianization functor
$\AbFunc\from\BurnsideCat\to\Burn{\ZZ}$ which is the identity on
objects and sends a morphism $(A_{t,s})_{s\in S, t\in T}$ to the
matrix $(\#A_{t,s})_{s\in S,t\in T}\in\ZZ^{T\times S}$.
We are given a
diagram $G\in\Burn{\ZZ}^{\ChainCat{2}^n}$ which we wish to lift to a
diagram $Q\in\lax{\iBurn{S^N}}^{\ChainCat{2}^n}$. As we will see, it suffices to lift $G$ to a diagram
$D\from\ChainCat{2}^n\to\BurnsideCat$.

Since $\BurnsideCat$ is a weak $2$-category, we should first clarify
what we mean by a diagram in $\BurnsideCat$. A \emph{strictly unital
  lax $2$-functor}---henceforth just called a lax
functor---$D\from\ChainCat{2}^n\to\BurnsideCat$ consists of the
following data:
\begin{enumerate}[leftmargin=*,parsep=2pt]
\item A finite set $F(x)\in\BurnsideCat$ for each
  $v\in\ChainCat{2}^n$.
\item An $F(v)\times F(u)$-matrix of finite sets $F(u\to
  v)\in\Hom_{\BurnsideCat}(F(u),F(v))$ for each $u>v\in \ChainCat{2}^n$.
\item A $2$-isomorphism $F_{u,v,w}\from F(v\to w)\circ_1 F(u\to v)\to F(u\to w)$ for each $u>v>w\in \ChainCat{2}^n$ so that
  for each $u>v>w>z$, $F_{u,w,z}\circ_2 (\Id\circ_1
  F_{u,v,w})=(F_{v,w,z}\circ_1\Id)\circ_2 F_{u,v,z}$, where $\circ_i$
  denotes composition of $i$-morphisms ($i=1,2$).
\end{enumerate}

Next we turn such a lax diagram
$D\from\ChainCat{2}^n\to\BurnsideCat$ into a lax diagram
$Q\in\lax{\iBurn{S^N}}^{\ChainCat{2}^n}$, $N\geq n+1$, satisfying
$\wt{H}_N\circ Q=\AbFunc\circ D$. Associate a \emph{box}
$B_x=[-1,1]^N$ to each $x\in D(v)$, $v\in\ChainCat{2}^n$. For
each $u>v$, let $D(u\to v)=(A_{y,x})_{x\in D(u),y\in D(v)}$ and let
$E(u\to v)$ be the space of embeddings
$\iota_{u,v}=\{\iota_{u,v,x}\}_{x\in D(u)}$ where
\begin{equation}
\iota_{u,v,x}\from \coprod_yA_{y,x}\times B_y\into B_x
\end{equation}
whose restriction to each copy of $B_y$ is a sub-box inclusion, i.e.,
composition of a translation and dilation. The space $E(A)$ is $(N-2)$-connected. 

For any such data $\iota_{u,v}$, a collapse map and a fold
map give a \emph{box map}
\begin{equation}
  \hat{\iota}_{u,v}\from\!\!\bigvee_{x\in D(u)}\!\!S^N=\!\!\coprod_{x\in
    D(u)}\!\!B_x/\bdy
  \to \!\!\coprod_{\substack{x\in D(u)\\y\in D(v)\\a\in A_{y,x}}}\!\!B_y/\bdy  
  \to
  \!\!\coprod_{y\in D(v)}\!\!B_y/\bdy=\!\!\bigvee_{y\in D(v)}\!\!S^N.
\end{equation}
Given $u>v>w$ and data $\iota_{u,v}$, $\iota_{v,w}$, the composition
$\hat{\iota}_{v,w}\circ\hat{\iota}_{u,v}$ is also a box map
corresponding to some induced embedding data.

The construction of the lax diagram
$Q\in\lax{\iBurn{S^N}}^{\ChainCat{2}^n}$ is inductive. On objects, $Q$
agrees with $D$. For (non-identity) morphisms $u\to v$, choose a box
map $Q(u\to v)=\hat{\iota}_{u,v}\from\bigvee_{x\in D(u)}S^N\to\bigvee_{y\in D(v)}S^N$ 
refining $D(u\to v)$. Staying in the space of box maps,
the required homotopies exist and are unique up to homotopy because
each $E(A)$ is $N-2\geq n-1$ connected, and there are no sequences of
composable morphisms of length
$>n-1$. (See~\cite{LLS-khovanov-product} for details.)

The above construction closely follows the Pontryagin-Thom procedure
from \S\ref{sec:flow-cat}. 
Indeed, functors from the cube to
the Burnside category correspond to certain kinds of flow categories
(\emph{cubical} ones), and the realizations in terms of box maps and
cubical flow categories agree.

\section{Khovanov  homology}
\subsection{The Khovanov cube}
Khovanov homology was defined by~\cite{Kho-kh-categorification} using the Frobenius algebra
$V=H^*(S^2)$. Let $x_-\in H^0(S^2)$ and $x_+\in H^2(S^2)$ be the
positive generators. (Our labeling is opposite Khovanov's convention,
as the maps in our cube go from $1$ to $0$.)
Via the equivalence of Frobenius algebras and
$(1+1)$-dimensional topological field theories (cf.~\cite{Abrams-top-TQFT}),
we can reinterpret $V$ as a functor from the $(1+1)$-dimensional
bordism category $\Cob^{1+1}$ to $\Burn{\ZZ}$
that assigns $\{x_+,x_-\}$ to circle, and hence $\prod_{\pi_0(C)}\{x_+,x_-\}$ to a one-manifold $C$. For $x\in V(C)$, let $\norm{x}_+$ (respectively, $\norm{x}_-$) denote the number of circles in $C$ labeled $x_+$ (respectively, $x_-$) by $x$.
For a cobordism $\Sigma\from C_1\to C_0$,
the map $V(\Sigma)\from
\ZZ\basis{V(C_1)}=\otimes_{\pi_0(C_1)}\ZZ\basis{x_+,x_-}\to
\ZZ\basis{V(C_0)}= \otimes_{\pi_0(C_0)}\ZZ\basis{x_+,x_-}$ is the tensor
product of the maps induced by the connected components of $\Sigma$;
and if $\Sigma\from C_1\to C_0$ is a connected, genus-$g$ cobordism, then
the $(y,x)$-entry of the matrix representing the map, $x\in V(C_1)$,
$y\in V(C_0)$, is
\begin{equation}
\begin{cases}
1&\text{if $g=0$, $\norm{x}_++\norm{y}_-=1$,}\\
2&\text{if $g=1$, $\norm{x}_+=\norm{y}_-=0$,}\\
0&\text{otherwise}
\end{cases}
\end{equation}
(cf.~\cite{Bar-kh-tangle-cob,HKK-Kh-htpy}).

Now, given a link diagram $L$ with
$n$ crossings numbered $c_1,\dots,c_n$, \cite{Kho-kh-categorification} constructs a cubical
diagram
$G_\Kh=V\circ\LL\in\Burn{\ZZ}^{\ChainCat{2}^n}$ where $\LL\from\ChainCat{2}^n\to\Cob^{1+1}$ is the \emph{cube of resolutions} (extending \cite{Kau-knot-resolutions})
defined as follows. For $v\in\ChainCat{2}^n$, let
$\LL(v)$ be the complete resolution of the link diagram $L$ formed by
resolving the $i\th$ crossing $c_i$
$\vcenter{\hbox{\begin{tikzpicture}[scale=0.04]
\draw (0,10) -- (10,0);
\node[crossing] at (5,5) {};
\draw (0,0) -- (10,10);
\end{tikzpicture}}}$
by the \emph{$0$-resolution}
$\vcenter{\hbox{\begin{tikzpicture}[scale=0.04]
\draw (0,0) .. controls (4,4) and (4,6) .. (0,10);
\draw (10,0) .. controls (6,4) and (6,6) .. (10,10);
\end{tikzpicture}}}$
if $v_i=0$ and by the \emph{$1$-resolution}
$\vcenter{\hbox{\begin{tikzpicture}[scale=0.04]
\draw (0,0) .. controls (4,4) and (6,4) .. (10,0);
\draw (0,10) .. controls (4,6) and (6,6) .. (10,10);
\end{tikzpicture}}}$
if $v_i=1$. For a morphism $u\to v$, $\LL(u\to v)$ is the cobordism which
is an elementary saddle from the $1$-resolution to the $0$-resolution
near crossings $c_i$ for each $i$ with $u_i>v_i$, and is a product
cobordism elsewhere.

The dual of the resulting total complex, shifted by
$n_-$, the number of negatives
crossings $\vcenter{\hbox{\begin{tikzpicture}[scale=0.04] \draw[->] (0,10) -- (10,0);
  \node[crossing] at (5,5) {}; \draw[->] (0,0) -- (10,10);
\end{tikzpicture}}}$
in $L$, is usually called the Khovanov complex
\begin{equation}
\KhCx^*(L)=\mathrm{Dual}(\tot(G_\Kh))[n_-],
\end{equation}
and its cohomology $\Kh^*(L)$ the Khovanov homology, which is a link
invariant. There is an internal grading, the \emph{quantum
  grading}, that comes from placing the two symbols $x_+$ and $x_-$ in
two different quantum gradings, and the entire complex decomposes along
this grading, so Khovanov homology inherits a second grading
$\Kh^i(L)=\oplus_j\Kh^{i,j}(L)$, and its quantum-graded Euler
characteristic
\begin{equation}
\sum_{i,j}(-1)^iq^j\mathrm{rank}(\Kh^{i,j}(L))
\end{equation}
recovers the unnormalized Jones polynomial of $L$.
The quantum grading persists in the space-level refinement but, for
brevity, we suppress it.

\subsection{The stable homotopy type}
Following \S\ref{sec:cube-and-Burn}, to give a space-level
refinement of Khovanov homology it suffices to lift $G_\Kh$ to a lax
functor $\ChainCat{2}^n\to\BurnsideCat$.

\cite[\S 3.2]{HKK-Kh-htpy} shows that the TQFT
$V\from\Cob^{1+1}\to\Burn{\ZZ}$ does not lift to a functor
$\Cob^{1+1}\to\BurnsideCat$.  However, we may instead work with the
\emph{embedded} cobordism category $\Cob_e^{1+1}$, which is a weak
$2$-category whose objects are closed $1$-manifolds embedded in $S^2$,
morphisms are compact cobordisms embedded in $S^2\times[0,1]$, and
$2$-morphisms are isotopy classes of isotopies in $S^2\times[0,1]$ rel
boundary. The cube $\LL\from\ChainCat{2}^n\to\Cob^{1+1}$ factors
through a functor $\LL_e\from\ChainCat{2}^n\to\Cob_e^{1+1}$. (This
functor $\LL_e$ is lax, similar to what we had for functors to the
Burnside category except without strict unitarity.)  So it remains to
lift $V$ to a (lax) functor $V_e\from\Cob^{1+1}_e\to\BurnsideCat$
\begin{equation}
\mathcenter{
\begin{tikzpicture}[xscale=3,yscale=0.6]
\node (cube) at (0,0) {$\ChainCat{2}^n$};
\node (cob) at (1,-1) {$\Cob^{1+1}$};
\node (cobe) at (1,1) {$\Cob^{1+1}_e$};
\node (burnz) at (2,-1) {$\Burn{\ZZ}$};
\node (burn) at (2,1) {$\BurnsideCat$.};

\draw[->] (cube) to node[below] {\scriptsize $\LL$} (cob); 
\draw[->] (cube) to node[above] {\scriptsize $\LL_e$} (cobe);
\draw[->] (cob) -- (burnz) node[midway, anchor=south] {\scriptsize $V$};
\draw[->, dashed] (cobe) -- (burn) node[midway,anchor=south] {\scriptsize $V_e$};
\draw[->] (cobe) -- (cob);
\draw[->] (burn) -- (burnz);
\end{tikzpicture}}
\end{equation}

On an embedded one-manifold $C$, we must set
\begin{equation}
V_e(C)=\prod_{\pi_0(C)}\{x_+,x_-\}.
\end{equation}
For an embedded cobordism $\Sigma\from C_1\to C_0$ with
$C_i$ embedded in $S^2\times\{i\}$, 
the matrix $V_e(\Sigma)$ is a tensor product over the connected
components of $\Sigma$, i.e., if $\Sigma=\amalg_{j=1}^m\bigl(\Sigma_j\co
C_{1,j}\to C_{0,j}\bigr)$ and $(y^j,x^j)\in V_e(C_{0,j})\times V_e(C_{1,j})$, then the $(\Pi_{j=1}^m y^j,\Pi_{j=1}^m x^j)$ entry of $V_e(\Sigma)$ equals
\begin{equation}
V_e(\Sigma_1)_{y^1,x^1}\times\cdots\times V_e(\Sigma_m)_{y^m,x^m}.
\end{equation}
And finally, if $\Sigma\from C_1\to C_0$ is a connected genus-$g$
cobordism, then for $x\in V_e(C_1)$ and $y\in V_e(C_0)$, the
$(y,x)$-entry of $V_e(\Sigma)$
must be a
\begin{equation}
\begin{cases}
1\text{-element set}&\text{if $g=0$, $\norm{x}_++\norm{y}_-=1$,}\\
2\text{-element set}&\text{if $g=1$, $\norm{x}_+=\norm{y}_-=0$,}\\
\emptyset&\text{otherwise.}
\end{cases}
\end{equation}
One-element sets do not have any non-trivial automorphisms, so we may
set all the one-element sets to $\{\pt\}$. The two-element sets
must be chosen carefully: they have to behave naturally under isotopy
of cobordisms (the $2$-morphisms in $\Cob^{1+1}_e$) and must admit
natural isomorphisms $V_e(\Sigma\circ\Sigma')\cong
V_e(\Sigma)\circ V_e(\Sigma')$ when composing cobordisms
$\Sigma'\from C_2\to C_1$ and $\Sigma\from C_1\to C_0$.

Decompose $S^2\times[0,1]$ as a union of
two compact $3$-manifolds glued along $\Sigma$, $A\cup_\Sigma B$. Set
$V_e(\Sigma)$ to be the (cardinality two) set of unordered bases
$\{\alpha,\beta\}$ for $\ker(H^1(\Sigma)\to H^1(\bdy\Sigma))\cong\ZZ^2$
so that $\alpha$ (respectively, $\beta$) is the restriction of a
generator of $\ker(H^1(A)\to H^1(A\cap(S^2\times\{0,1\})))\cong\ZZ$
(respectively, $\ker(H^1(B)\to H^1(B\cap(S^2\times\{0,1\})))\cong\ZZ$),
and so that, if we orient $\Sigma$ as the boundary of $A$ then
$\langle \alpha\cup\beta,[\Sigma]\rangle=1$ (or equivalently, if we
orient $\Sigma$ as the boundary of $B$ then $\langle
\beta\cup\alpha,[\Sigma]\rangle=1$).  This
assignment is clearly natural.

Given cobordisms $\Sigma'\from C_2\to C_1$ and $\Sigma\from C_1\to
C_0$, we need to construct a natural $2$-isomorphism $V_e(\Sigma)\circ
V_e(\Sigma')\to V_e(\Sigma\circ\Sigma')$. The only non-trivial case is
when $\Sigma$ and $\Sigma'$ are genus-$0$ cobordisms gluing to form a
connected, genus-$1$ cobordism. In that case, letting $x\in V_e(C_2)$
(respectively, $y\in V_e(C_0)$) denote the generator that labels all
circles of $C_2$ by $x_-$ (respectively, all circles of $C_0$ by
$x_+$), we need to construct a bijection between the $(y,x)$-entry
$M_{y,x}$ of $V_e(\Sigma)\circ V_e(\Sigma')$ and the $(y,x)$-entry
$N_{y,x}$ of $V_e(\Sigma\circ\Sigma')$. Consider an element $Z$ of
$M_{y,x}$; $Z$ specifies an element $z\in V(C_1)$. There is a unique
circle $C$ in $C_1$ that is non-separating in $\Sigma\circ\Sigma'$ and
is labeled $x_+$ by $z$. Choose an orientation $o$ of
$\Sigma\circ\Sigma'$, orient $C$ as the boundary of $\Sigma$, and let
$[C]$ denote the image of $C$ in
$H_1(\Sigma\circ\Sigma',\bdy(\Sigma\circ\Sigma'))$.  Assign to $Z$ the
unique basis in $N_{y,x}$ that contains the Poincar\'e dual of $[C]$.
It is easy to check that this map is well-defined, independent of the
choice of $o$, natural, and a bijection.

This concludes the definition of the functor
$V_e\from\Cob^{1+1}_e\to\BurnsideCat$. The spatial lift $Q_\Kh\in
\lax{\iBurn{S^N}}^{\ChainCat{2}^n}$ is then induced from the
composition
$V_e\circ\LL_e\from\ChainCat{2}^n\to\BurnsideCat$. Totalization
produces a cell complex $\tot(Q_\Kh)$ with
\begin{equation}
\cocellC[*](\tot(Q_\Kh))[N+n_-]=\KhCx^*(L).
\end{equation} 
We define the Khovanov spectrum $\KhSpace(L)$ to be the
formal $(N+n_-)\th$ desuspension of $\tot(Q_{\Kh})$.  The stable
homotopy type of $\KhSpace(L)$ is a link invariant; see \cite{RS-khovanov,HKK-Kh-htpy,LLS-khovanov-product}. 
The spectrum decomposes as a wedge sum over quantum
gradings, $\KhSpace(L)=\bigvee_j\KhSpace^j(L)$.
There is also a reduced version of the Khovanov stable homotopy type, $\rKhSpace(L)$, refining the reduced Khovanov chain complex.

\subsection{Properties and applications}\label{sec:prop-app}
In order to apply the Khovanov homotopy type to knot theory, one needs
to extract some concrete information from it beyond Khovanov
homology. Doing so, one encounters three difficulties:
\begin{enumerate}[leftmargin=*,label=\protect\Sadey\arabic*.,ref=\protect\Sadey\arabic*,parsep=2pt]
\item The number of vertices of the Khovanov cube is $2^n$, where $n$
  is the number of crossings of $L$, so the number of cells in the CW
  complex $\KhSpace(L)$ grows at least that fast. So, direct
  computation must be by computer, and for relatively low crossing
  number links.
\item\label{item:low-crossing} For low crossing number links, $\Kh^{i,j}(L)$
  is supported near the diagonal $2i-j=\sigma(L)$, so each
  $\KhSpace^j(L)$ has nontrivial homology only in a small number of
  adjacent gradings, and these $\Kh^{i,j}(L)$ have no $p$-torsion for $p>2$. If $X$ is a
  spectrum so that $\wt{H}^i(X)$ is nontrivial only for
  $i\in\{k,k+1\}$ then the homotopy type of $X$ is determined by
  $\wt{H}^*(X)$, while if $\wt{H}^*(X)$ is nontrivial in only three
  adjacent gradings and has no $p$-torsion ($p>2$) then the homotopy
  type of $X$ is determined by $\wt{H}^*(X)$ and the Steenrod operations
  $\Sq^1$ and $\Sq^2$ (see \cite[Theorems 11.2, 11.7]{Baues-top-handbook}).
\item There are no known formulas for most algebro-topological
  invariants of a CW complex.
  (The situation is a bit better for simplicial complexes.)
\end{enumerate}

\cite{RS-steenrod} found an explicit formula for the operation
$\Sq^2\co
\Kh^{i,j}(L;\FF_2)\to\Kh^{i+2,j}(L;\FF_2)$.
The
operation $\Sq^1$ is the Bockstein, and hence easy to compute. Using
these, one can determine the spectra $\KhSpace^j(L)$ for all prime
links up to $11$ crossings. All these spectra are wedge sums of
(de)suspensions of 6 basic pieces (cf.~\ref{item:low-crossing}),
and all possible basic pieces except $\CC P^2$ occur (see~\ref{Q:CP2}). The first knot for which
$\KhSpace^{i,j}(K)$ is not a Moore space is also the first
non-alternating knot: $T(3,4)$. Extending these computations:
\begin{theorem}
  \cite{Seed-Kh-square} There are pairs of knots with isomorphic
  Khovanov cohomologies but non-homotopy equivalent Khovanov spectra.
\end{theorem}
The first such pair is $11^n_{70}$ and $13^n_{2566}$. \cite{JLS-Kh-moves} introduced moves and simplifications allowing them to give a by-hand computation of $\Sq^2$ for $T(3,4)$ and some other knots.

\begin{theorem}\cite{LLS-khovanov-product}
  Given links $L$, $L'$, $\KhSpace(L\amalg L')\simeq
  \KhSpace(L)\smas\KhSpace(L')$ and, if $L$ and $L'$ are based,
  $\rKhSpace(L\# L')\simeq \rKhSpace(L)\smas\rKhSpace(L')$ and
  $\KhSpace(L\# L')\simeq \KhSpace(L)\smas_{\KhSpace(U)}\KhSpace(L')$. Finally, if
  $m(L)$ is the mirror of $L$ then $\KhSpace(m(L))$ is the
  Spanier-Whitehead dual to $\KhSpace(L)$.
\end{theorem}

\begin{corollary} 
  For any integer $k$ there is a knot $K$ so that the operation
  $\Sq^k\from\Kh^{*,*}(K)\to \Kh^{*+k,*}(K)$ is nontrivial. (Compare~\ref{Q:Sqk}.)
\end{corollary}
\begin{proof}
  Choose a knot $K_0$ so that in some quantum grading, $\rKh(K_0)$ has
  $2$-torsion but $\rKh(K_0;\FF_2)$ has vanishing $\Sq^i$ for
  $i>1$. (For instance, $K_0=13^n_{3663}$
  works, \cite{Shu-kh-torsion}.) Let
  $K=\overbrace{K_0\#\cdots\#K_0}^k$. By the Cartan formula,
  $\Sq^k(\alpha)\neq 0$ for some $\alpha\in \rKh(K;\FF_2)$. The
  short exact sequence
  \[
    0\to \rKh(K;\FF_2)\to
    \Kh(K;\FF_2)\to
    \rKh(K;\FF_2)\to 0
  \]from \cite[\S4.3]{Ras-kh-survey} is induced by a cofiber sequence of Khovanov spectra from \cite[\S8]{RS-khovanov}, so if
  $\beta\in \Kh(K;\FF_2)$ is any preimage of
  $\alpha$ then by naturality, $\Sq^k(\beta)\neq 0$, as well.
\end{proof}

\cite{Plam-transKh} defined an invariant of links $L$ in $S^3$ transverse to
the standard contact structure, as 
an element of the Khovanov homology of $L$.
\begin{theorem}\cite{LNS-Plam}
  Given a transverse link $L$ in $S^3$ there is a well-defined cohomotopy class of $\KhSpace(L)$ lifting Plamenevskaya's invariant.
\end{theorem}
While \cite{LNS-Plam} show that Plamenevskaya's class is known to be invariant under flypes, the homotopical refinement is not presently known to be. It remains open whether either invariant is effective (i.e., stronger than the self-linking number).

The Steenrod squares on Khovanov homology was used by \cite{RS-s-invariant} to tweak
the concordance invariant and slice-genus bound
$s$ by \cite{Ras-kh-slice} to give potentially new concordance invariants
and slice genus bounds. In the simplest case,
$\Sq^2$, these concordance invariants are,
indeed, different from Rasmussen's
invariants. They can be used to give some new
results on the 4-ball genus for certain families of
knots, see \cite{LLS-khovanov-product}. More striking, \cite{FLL}
used these operations to resolve whether certain knots are
\emph{squeezed}, i.e., occur in a minimal-genus cobordism between
positive and negative torus knots.

In a different direction, the Khovanov homotopy type admits a number
of extensions. \cite{LOS-Kh-colored} and, independently, \cite{Willis-Kh-tail} proved
that the Khovanov homotopy type stabilizes under adding twists, and
used this to extend it to a colored Khovanov stable
homotopy type; further
stabilization results were proved by \cite{Willis-Kh-tail} and
\cite{GW-Kh-infinite}. \cite{JLS-Kh-sln} proposed
a homotopical refinement of the $\mathfrak{sl}_n$ Khovanov-Rozansky
homology for a large class of knots and there is
also work in progress in this direction by \cite{HKS-Kh-sln}. \cite{SSS-odd-kh-homotopy-type} gave a homotopical refinement of the odd Khovanov homology of \cite{OSzR-kh-oddkhovanov}.

The construction of the functor $V_e$ is natural enough that it was used by \cite{LLS-tangles} to give a space-level refinement of the arc algebras and tangle invariants from \cite{Kho-kh-tangles}. In the refinement, the arc algebras are replaced by ring spectra (or, if one prefers, spectral categories), and the tangle invariants by module spectra. 

\subsection{Speculation}\label{sec:speculation}
We conclude with some open questions:
\begin{enumerate}[leftmargin=*,label=\protect\Quesy\arabic*.,ref=\protect\Quesy\arabic*,parsep=2pt]
\item\label{Q:CP2} Does $\CC P^2$ occur as a wedge summand of the Khovanov spectrum associated to some link? (Cf.~\S\ref{sec:prop-app}.)
  More generally, are there non-obvious restrictions on the spectra which occur in the Khovanov homotopy types?
\item Is the obstruction to amphichirality coming from the Khovanov spectrum stronger than the obstruction coming from Khovanov homology? Presumably the answer is ``yes,'' but verifying this might require interesting new computational techniques.
\item\label{Q:Sqk} Are there prime knots with arbitrarily high
  Steenrod squares? Other power operations? Again, we expect that the
  answer is ``yes.'' 
\item How can one compute Steenrod operations, or stable homotopy
  invariants beyond homology, from a flow category?
  (Compare~\cite{RS-steenrod}.)
\item\label{Q:plam} Is the refined Plamenevskaya invariant from~\cite{LNS-Plam} effective? Alternatively, is it invariant under negative flypes / $SZ$ moves?
\item Is there a well-defined homotopy class of maps of Khovanov
  spectra associated to an isotopy class of link cobordisms $\Sigma\subset [0,1]\times
  \RR^3$? Given such a cobordism $\Sigma$ in general position with
  respect to projection to $[0,1]$, there is an
  associated map, but it is not known if this map
  is an isotopy invariant. More generally, one could hope to associate
  an $(\infty,1)$-functor from a quasicategory of links and embedded
  cobordisms to a quasicategory of spectra, allowing one
  to study families of cobordisms. If not, this is a sense in which
  Khovanov homotopy, or perhaps homology, is \emph{un}natural.  Applications of these cobordism
  maps would also be interesting (cf.~\cite{Swann-kh}).
\item If analytic difficulties are resolved, applying the
  Cohen-Jones-Segal construction to the symplectic Khovanov
  homology of \cite{SS-Kh-symp} should also give a Khovanov spectrum. Is
  that symplectic Khovanov spectrum homotopy equivalent to the
  combinatorial Khovanov spectrum? (Cf.~\cite{AS-kh-agree-Q}.)
\item The (symplectic) Khovanov complex admits, in some sense, an
  $O(2)$-action, cf.~\cite{Man-Kh-symp,SS-kh-local,HLS-gauge-Lie,SSS-geometric-perturb}. Does
  the Khovanov stable homotopy type?
\item Is there a homotopical refinement of the \cite{Lee-kh-endomorphism} or \cite{Bar-kh-tangle-cob}
  deformation of Khovanov homology? Perhaps no genuine spectrum
  exists, but one can hope to find a lift of the theory to a module
  over $\mathit{ku}$ or $\mathit{ko}$ or another ring spectrum
  (cf.~\cite{Cohen-gauge-realize}). Exactly how far one can lift the
  complex might be predicted by the polarization class of a partial
  compactification of the symplectic Khovanov
  setting from \cite{SS-Kh-symp}.
\item Can one make the discussion in \S\ref{sec:infty} precise? Are there other rig (or $\infty$-rig)
  categories, beyond $\mathsf{Sets}$, useful in refining chain complexes in
  categorification or Floer theory to get modules over appropriate
  ring spectra?
\item Is there an intrinsic, diagram-free description of
  $\KhSpace(K)$ or, for that matter, for Khovanov homology or the
  Jones polynomial?
\end{enumerate}

\bibliographystyle{apalike}
\bibliography{shortbibfile}

\begin{thebibliography}{}

\bibitem[Abouzaid and Kragh, 2016]{AK-gauge-immersion}
Abouzaid, M. and Kragh, T. (2016).
\newblock On the immersion classes of nearby {L}agrangians.
\newblock {\em J. Topol.}, 9(1):232--244.

\bibitem[Abouzaid and Smith, ]{AS-kh-agree-Q}
Abouzaid, M. and Smith, I.
\newblock {K}hovanov homology from {F}loer cohomology.
\newblock arXiv:1504.01230.

\bibitem[Abrams, 1996]{Abrams-top-TQFT}
Abrams, L. (1996).
\newblock Two-dimensional topological quantum field theories and {F}robenius
  algebras.
\newblock {\em J. Knot Theory Ramifications}, 5(5):569--587.

\bibitem[Bar-Natan, 2002]{Bar-kh-khovanovs}
Bar-Natan, D. (2002).
\newblock On {K}hovanov's categorification of the {J}ones polynomial.
\newblock {\em Algebr. Geom. Topol.}, 2:337--370 (electronic).

\bibitem[Bar-Natan, 2005]{Bar-kh-tangle-cob}
Bar-Natan, D. (2005).
\newblock Khovanov's homology for tangles and cobordisms.
\newblock {\em Geom. Topol.}, 9:1443--1499.

\bibitem[Bauer, 2004]{Bauer-gauge-BF}
Bauer, S. (2004).
\newblock A stable cohomotopy refinement of {S}eiberg-{W}itten invariants.
  {II}.
\newblock {\em Invent. Math.}, 155(1):21--40.

\bibitem[Bauer and Furuta, 2004]{BF-gauge-BF}
Bauer, S. and Furuta, M. (2004).
\newblock A stable cohomotopy refinement of {S}eiberg-{W}itten invariants. {I}.
\newblock {\em Invent. Math.}, 155(1):1--19.

\bibitem[Baues, 1995]{Baues-top-handbook}
Baues, H.~J. (1995).
\newblock Homotopy types.
\newblock In {\em Handbook of algebraic topology}, pages 1--72. North-Holland,
  Amsterdam.

\bibitem[Bott, 1980]{Bott-top-history}
Bott, R. (1980).
\newblock Marston {M}orse and his mathematical works.
\newblock {\em Bull. Amer. Math. Soc. (N.S.)}, 3(3):907--950.

\bibitem[Bousfield and Kan, 1972]{BK-top-book}
Bousfield, A. and Kan, D. (1972).
\newblock {\em Homotopy limits, completions and localizations}.
\newblock Lecture Notes in Mathematics, Vol. 304. Springer-Verlag, Berlin.

\bibitem[Carlsson, 1981]{Carlsson-top-counter}
Carlsson, G. (1981).
\newblock A counterexample to a conjecture of {S}teenrod.
\newblock {\em Invent. Math.}, 64(1):171--174.

\bibitem[Cohen, 2009]{Cohen-gauge-realize}
Cohen, R. (2009).
\newblock Floer homotopy theory, realizing chain complexes by module spectra,
  and manifolds with corners.
\newblock In {\em Algebraic topology}, volume~4 of {\em Abel Symp.}, pages
  39--59. Springer, Berlin.

\bibitem[Cohen, 2010]{Cohen10:Floer-htpy-cotangent}
Cohen, R. (2010).
\newblock The {F}loer homotopy type of the cotangent bundle.
\newblock {\em Pure Appl. Math. Q.}, 6(2, Special Issue: In honor of Michael
  Atiyah and Isadore Singer):391--438.

\bibitem[Cohen et~al., 1995]{CJS-gauge-floerhomotopy}
Cohen, R., Jones, J., and Segal, G. (1995).
\newblock Floer's infinite-dimensional {M}orse theory and homotopy theory.
\newblock In {\em The {F}loer memorial volume}, volume 133 of {\em Progr.
  Math.}, pages 297--325. Birkh\"auser, Basel.

\bibitem[Cordier, 1982]{Cordier}
Cordier, J.-M. (1982).
\newblock Sur la notion de diagramme homotopiquement coh\'erent.
\newblock {\em Cahiers Topologie G\'eom. Diff\'erentielle}, 23(1):93--112.
\newblock Third Colloquium on Categories, Part VI (Amiens, 1980).

\bibitem[Douglas, ]{Douglas-top-spectre}
Douglas, C.
\newblock Twisted parametrized stable homotopy theory.
\newblock arXiv:math/0508070.

\bibitem[Elmendorf and Mandell, 2006]{EM-top-machine}
Elmendorf, A. and Mandell, M. (2006).
\newblock Rings, modules, and algebras in infinite loop space theory.
\newblock {\em Adv. Math.}, 205(1):163--228.

\bibitem[Everitt et~al., 2016]{ELST-trivial}
Everitt, B., Lipshitz, R., Sarkar, S., and Turner, P. (2016).
\newblock Khovanov homotopy types and the {D}old-{T}hom functor.
\newblock {\em Homology Homotopy Appl.}, 18(2):177--181.

\bibitem[Everitt and Turner, 2014]{ET-kh-spectrum}
Everitt, B. and Turner, P. (2014).
\newblock The homotopy theory of {K}hovanov homology.
\newblock {\em Algebr. Geom. Topol.}, 14(5):2747--2781.

\bibitem[Feller et~al., ]{FLL}
Feller, P., Lewark, L., and Lobb, A.
\newblock In preparation.

\bibitem[Floer, 1988a]{Floer-gauge-instanton}
Floer, A. (1988a).
\newblock An instanton-invariant for {$3$}-manifolds.
\newblock {\em Comm. Math. Phys.}, 118(2):215--240.

\bibitem[Floer, 1988b]{Floer-top-Lagrangian}
Floer, A. (1988b).
\newblock Morse theory for {L}agrangian intersections.
\newblock {\em J. Differential Geom.}, 28(3):513--547.

\bibitem[Floer, 1988c]{Floer-top-unregularized}
Floer, A. (1988c).
\newblock The unregularized gradient flow of the symplectic action.
\newblock {\em Comm. Pure Appl. Math.}, 41(6):775--813.

\bibitem[Furuta, 2001]{Furuta-gauge-BF}
Furuta, M. (2001).
\newblock Monopole equation and the {$\frac{11}8$}-conjecture.
\newblock {\em Math. Res. Lett.}, 8(3):279--291.

\bibitem[Hendricks et~al., ]{HLS-gauge-Lie}
Hendricks, K., Lipshitz, R., and Sarkar, S.
\newblock A simplicial construction of {$G$}-equivariant {F}loer homology.
\newblock arXiv:1609.09132.

\bibitem[Hu et~al., 2016]{HKK-Kh-htpy}
Hu, P., Kriz, D., and Kriz, I. (2016).
\newblock Field theories, stable homotopy theory and {K}hovanov homology.
\newblock {\em Topology Proc.}, 48:327--360.

\bibitem[Hu et~al., ]{HKS-Kh-sln}
Hu, P., Kriz, I., and Somberg, P.
\newblock Derived representation theory and stable homotopy categorification of
  {$\mathit{sl}_k$}.
\newblock Currently available at
  {\url{http://www.math.lsa.umich.edu/~ikriz/drt16084.pdf}}.

\bibitem[Islambouli and Willis, ]{GW-Kh-infinite}
Islambouli, G. and Willis, M.
\newblock The {K}hovanov homology of infinite braids.
\newblock arXiv:1610.04582.

\bibitem[J{\"a}nich, 1968]{Jan-top-o(n)manifolds}
J{\"a}nich, K. (1968).
\newblock On the classification of {$O(n)$}-manifolds.
\newblock {\em Math. Ann.}, 176:53--76.

\bibitem[Jones et~al., a]{JLS-Kh-sln}
Jones, D., Lobb, A., and Sch\"utz, D.
\newblock An {$\mathfrak{sl}_n$} stable homotopy type for matched diagrams.
\newblock arXiv:1506.07725.

\bibitem[Jones et~al., b]{JLS-Kh-moves}
Jones, D., Lobb, A., and Sch\"utz, D.
\newblock {M}orse moves in flow categories.
\newblock arXiv:1507.03502.

\bibitem[Jones, 1985]{Jones-polynomial}
Jones, V. (1985).
\newblock A polynomial invariant for knots via von {N}eumann algebras.
\newblock {\em Bull. Amer. Math. Soc. (N.S.)}, 12(1):103--111.

\bibitem[Kauffman, 1987]{Kau-knot-resolutions}
Kauffman, L. (1987).
\newblock State models and the {J}ones polynomial.
\newblock {\em Topology}, 26(3):395--407.

\bibitem[Khandhawit, 2015a]{Khandhawit-gauge-slice}
Khandhawit, T. (2015a).
\newblock A new gauge slice for the relative {B}auer-{F}uruta invariants.
\newblock {\em Geom. Topol.}, 19(3):1631--1655.

\bibitem[Khandhawit, 2015b]{Khandhawit-gauge-Conley}
Khandhawit, T. (2015b).
\newblock On the stable {C}onley index in {H}ilbert spaces.
\newblock {\em J. Fixed Point Theory Appl.}, 17(4):753--773.

\bibitem[Khandhawit et~al., ]{TLS-gauge-unfolded}
Khandhawit, T., Lin, J., and Sasahira, H.
\newblock Unfolded {S}eiberg-{W}itten floer spectra, {I}: Definition and
  invariance.
\newblock arXiv:1604.08240.

\bibitem[Khovanov, 2000]{Kho-kh-categorification}
Khovanov, M. (2000).
\newblock A categorification of the {J}ones polynomial.
\newblock {\em Duke Math. J.}, 101(3):359--426.

\bibitem[Khovanov, 2002]{Kho-kh-tangles}
Khovanov, M. (2002).
\newblock A functor-valued invariant of tangles.
\newblock {\em Algebr. Geom. Topol.}, 2:665--741 (electronic).

\bibitem[Khovanov and Rozansky, 2008]{KR-kh-matrixfactorizations}
Khovanov, M. and Rozansky, L. (2008).
\newblock Matrix factorizations and link homology.
\newblock {\em Fund. Math.}, 199(1):1--91.

\bibitem[Khovanov and Seidel, 2002]{KS-kh-braid}
Khovanov, M. and Seidel, P. (2002).
\newblock Quivers, {F}loer cohomology, and braid group actions.
\newblock {\em J. Amer. Math. Soc.}, 15(1):203--271.

\bibitem[Kragh, ]{Kragh:transfer-spectra}
Kragh, T.
\newblock The {V}iterbo transfer as a map of spectra.
\newblock arXiv:0712.2533.

\bibitem[Kragh, 2013]{Kragh-gauge-param}
Kragh, T. (2013).
\newblock Parametrized ring-spectra and the nearby {L}agrangian conjecture.
\newblock {\em Geom. Topol.}, 17(2):639--731.
\newblock With an appendix by Mohammed Abouzaid.

\bibitem[Kronheimer and Manolescu, ]{KM-gauge-swspectrum}
Kronheimer, P. and Manolescu, C.
\newblock Periodic {F}loer pro-spectra from the {S}eiberg-{W}itten equations.
\newblock arXiv:math/0203243.

\bibitem[Laures, 2000]{Lau-top-cobordismcorners}
Laures, G. (2000).
\newblock On cobordism of manifolds with corners.
\newblock {\em Trans. Amer. Math. Soc.}, 352(12):5667--5688 (electronic).

\bibitem[Lawson et~al., a]{LLS-khovanov-product}
Lawson, T., Lipshitz, R., and Sarkar, S.
\newblock {K}hovanov homotopy type, {B}urnside category, and products.
\newblock arXiv:1505.00213.

\bibitem[Lawson et~al., b]{LLS-tangles}
Lawson, T., Lipshitz, R., and Sarkar, S.
\newblock {K}hovanov spectra for tangles.
\newblock arXiv:1706.02346.

\bibitem[Lawson et~al., 2017]{LLS-cube}
Lawson, T., Lipshitz, R., and Sarkar, S. (2017).
\newblock The cube and the {B}urnside category.
\newblock In {\em Categorification in geometry, topology, and physics}, volume
  684 of {\em Contemp. Math.}, pages 63--85. Amer. Math. Soc., Providence, RI.

\bibitem[Lee, 2005]{Lee-kh-endomorphism}
Lee, E.~S. (2005).
\newblock An endomorphism of the {K}hovanov invariant.
\newblock {\em Adv. Math.}, 197(2):554--586.

\bibitem[Lipshitz et~al., 2015]{LNS-Plam}
Lipshitz, R., Ng, L., and Sarkar, S. (2015).
\newblock On transverse invariants from {K}hovanov homology.
\newblock {\em Quantum Topol.}, 6(3):475--513.

\bibitem[Lipshitz and Sarkar, 2014a]{RS-khovanov}
Lipshitz, R. and Sarkar, S. (2014a).
\newblock A {K}hovanov stable homotopy type.
\newblock {\em J. Amer. Math. Soc.}, 27(4):983--1042.

\bibitem[Lipshitz and Sarkar, 2014b]{RS-s-invariant}
Lipshitz, R. and Sarkar, S. (2014b).
\newblock A refinement of {R}asmussen's {$s$}-invariant.
\newblock {\em Duke Math. J.}, 163(5):923--952.

\bibitem[Lipshitz and Sarkar, 2014c]{RS-steenrod}
Lipshitz, R. and Sarkar, S. (2014c).
\newblock A {S}teenrod square on {K}hovanov homology.
\newblock {\em J. Topol.}, 7(3):817--848.

\bibitem[Lipyanskiy, ]{Lipy-top-geom}
Lipyanskiy, M.
\newblock Geometric cycles in {F}loer theory.
\newblock arXiv:1409.1126.

\bibitem[Lobb et~al., 2017]{LOS-Kh-colored}
Lobb, A., Orson, P., and Sch\"utz, D. (2017).
\newblock A {K}hovanov stable homotopy type for colored links.
\newblock {\em Algebr. Geom. Topol.}, 17(2):1261--1281.

\bibitem[Lurie, 2009a]{Lurie-top-htt}
Lurie, J. (2009a).
\newblock {\em Higher topos theory}, volume 170 of {\em Annals of Mathematics
  Studies}.
\newblock Princeton University Press, Princeton, NJ.

\bibitem[Lurie, 2009b]{Lurie-top-cobordism}
Lurie, J. (2009b).
\newblock On the classification of topological field theories.
\newblock In {\em Current developments in mathematics, 2008}, pages 129--280.
  Int. Press, Somerville, MA.

\bibitem[Manolescu, 2003]{Man-gauge-swspectrum}
Manolescu, C. (2003).
\newblock Seiberg-{W}itten-{F}loer stable homotopy type of three-manifolds with
  {$b_1=0$}.
\newblock {\em Geom. Topol.}, 7:889--932 (electronic).

\bibitem[Manolescu, 2006]{Man-Kh-symp}
Manolescu, C. (2006).
\newblock Nilpotent slices, {H}ilbert schemes, and the {J}ones polynomial.
\newblock {\em Duke Math. J.}, 132(2):311--369.

\bibitem[Manolescu, 2007]{Man-kh-symp-sln}
Manolescu, C. (2007).
\newblock Link homology theories from symplectic geometry.
\newblock {\em Adv. Math.}, 211(1):363--416.

\bibitem[Milnor, 1963]{Milnor-top-Morse}
Milnor, J. (1963).
\newblock {\em Morse theory}.
\newblock Based on lecture notes by M. Spivak and R. Wells. Annals of
  Mathematics Studies, No. 51. Princeton University Press, Princeton, N.J.

\bibitem[Milnor, 1965]{Milnor-top-h}
Milnor, J. (1965).
\newblock {\em Lectures on the {$h$}-cobordism theorem}.
\newblock Notes by L. Siebenmann and J. Sondow. Princeton University Press,
  Princeton, N.J.

\bibitem[Morse, 1925]{Morse-top-fns}
Morse, M. (1925).
\newblock Relations between the critical points of a real function of {$n$}
  independent variables.
\newblock {\em Trans. Amer. Math. Soc.}, 27(3):345--396.

\bibitem[Morse, 1930]{Morse-top-Morse2}
Morse, M. (1930).
\newblock The foundations of a theory of the calculus of variations in the
  large in {$m$}-space. {II}.
\newblock {\em Trans. Amer. Math. Soc.}, 32(4):599--631.

\bibitem[Morse, 1996]{Morse-top-book}
Morse, M. (1996).
\newblock {\em The calculus of variations in the large}, volume~18 of {\em
  American Mathematical Society Colloquium Publications}.
\newblock American Mathematical Society, Providence, RI.
\newblock Reprint of the 1932 original.

\bibitem[Ozsv\'ath et~al., 2013]{OSzR-kh-oddkhovanov}
Ozsv\'ath, P., Rasmussen, J., and Szab\'o, Z. (2013).
\newblock Odd {K}hovanov homology.
\newblock {\em Algebr. Geom. Topol.}, 13(3):1465--1488.
\newblock arXiv:0710.4300.

\bibitem[Palais, 1963]{Palais-top-PS}
Palais, R. (1963).
\newblock Morse theory on {H}ilbert manifolds.
\newblock {\em Topology}, 2:299--340.

\bibitem[Palais and Smale, 1964]{PS-top-Bull}
Palais, R. and Smale, S. (1964).
\newblock A generalized {M}orse theory.
\newblock {\em Bull. Amer. Math. Soc.}, 70:165--172.

\bibitem[Pardon, ]{Pardon-top-CH}
Pardon, J.
\newblock Contact homology and virtual fundamental cycles.
\newblock arXiv:1508.03873.

\bibitem[Plamenevskaya, 2006]{Plam-transKh}
Plamenevskaya, O. (2006).
\newblock Transverse knots and {K}hovanov homology.
\newblock {\em Math. Res. Lett.}, 13(4):571--586.

\bibitem[Prasma et~al., ]{nLab-nonfunc}
Prasma, M. et~al.
\newblock {``Moore space: (Non-)Functoriality of the construction''} in {nLab}.
\newblock \url{ncatlab.org/nlab/show/Moore+space}.

\bibitem[Rasmussen, 2005]{Ras-kh-survey}
Rasmussen, J. (2005).
\newblock Knot polynomials and knot homologies.
\newblock In {\em Geometry and topology of manifolds}, volume~47 of {\em Fields
  Inst. Commun.}, pages 261--280. Amer. Math. Soc., Providence, RI.

\bibitem[Rasmussen, 2010]{Ras-kh-slice}
Rasmussen, J. (2010).
\newblock Khovanov homology and the slice genus.
\newblock {\em Invent. Math.}, 182(2):419--447.

\bibitem[Sarkar et~al., ]{SSS-odd-kh-homotopy-type}
Sarkar, S., Scaduto, C., and Stoffregen, M.
\newblock An odd {K}hovanov homotopy type.
\newblock arXiv:1801.06308.

\bibitem[Sarkar et~al., 2017]{SSS-geometric-perturb}
Sarkar, S., Seed, C., and Szab\'o, Z. (2017).
\newblock A perturbation of the geometric spectral sequence in {K}hovanov
  homology.
\newblock {\em Quantum Topol.}, 8(3):571--628.

\bibitem[Sasahira, ]{Sasahira-gauge-gluing}
Sasahira, H.
\newblock Gluing formula for the stable cohomotopy version of
  {S}eiberg-{W}itten invariants along 3-manifolds with {$b_1>0$}.
\newblock arXiv:1408.2623.

\bibitem[Seed, ]{Seed-Kh-square}
Seed, C.
\newblock Computations of the {L}ipshitz-{S}arkar {S}teenrod square on
  {K}hovanov homology.
\newblock arXiv:1210.1882.

\bibitem[Segal, 1974]{Segal-top-categories}
Segal, G. (1974).
\newblock Categories and cohomology theories.
\newblock {\em Topology}, 13:293--312.

\bibitem[Seidel and Smith, 2006]{SS-Kh-symp}
Seidel, P. and Smith, I. (2006).
\newblock A link invariant from the symplectic geometry of nilpotent slices.
\newblock {\em Duke Math. J.}, 134(3):453--514.

\bibitem[Seidel and Smith, 2010]{SS-kh-local}
Seidel, P. and Smith, I. (2010).
\newblock Localization for involutions in {F}loer cohomology.
\newblock {\em Geom. Funct. Anal.}, 20(6):1464--1501.

\bibitem[Shumakovitch, 2014]{Shu-kh-torsion}
Shumakovitch, A. (2014).
\newblock Torsion of {K}hovanov homology.
\newblock {\em Fund. Math.}, 225(1):343--364.

\bibitem[Smale, 1964]{Smale-top-PS}
Smale, S. (1964).
\newblock Morse theory and a non-linear generalization of the {D}irichlet
  problem.
\newblock {\em Ann. of Math. (2)}, 80:382--396.

\bibitem[Swann, 2010]{Swann-kh}
Swann, J. (2010).
\newblock {\em Relative {K}hovanov-{J}acobsson classes for spanning surfaces}.
\newblock ProQuest LLC, Ann Arbor, MI.
\newblock Thesis (Ph.D.)--Bryn Mawr College.

\bibitem[Vogt, 1973]{Vogt-top-hocolim}
Vogt, R. (1973).
\newblock Homotopy limits and colimits.
\newblock {\em Math. Z.}, 134:11--52.

\bibitem[Willis, ]{Willis-Kh-tail}
Willis, M.
\newblock A colored {K}hovanov homotopy type and its tail for {B}-adequate
  links.
\newblock arXiv:1602.03856.

\end{thebibliography}

\end{document}